\documentclass[12pt]{amsart}
\usepackage{amssymb}
\usepackage{amsthm}
\usepackage{amsmath}

\theoremstyle{plain}
\newtheorem{thm}{Theorem}[section]
\newtheorem{prop}[thm]{Proposition}
\newtheorem{cor}[thm]{Corollary}
\newtheorem{lem}[thm]{Lemma}

\theoremstyle{definition}

\newtheorem{ex}[thm]{Example}

\theoremstyle{remark}
\newtheorem{remark}[thm]{Remark}

\newcommand{\R}{\mathbb{R}}
\newcommand{\bR}{\mathbb{R}}

\newcommand{\Z}{\mathbb{Z}}

\newcommand{\F}{\mathbb{F}}

\newcommand{\bpf}{\begin{proof}}
\newcommand{\epf}{\end{proof}}
\newcommand{\bea}{\begin{eqnarray*}}
\newcommand{\eea}{\end{eqnarray*}}

\DeclareMathOperator{\Hom}{Hom}

\DeclareMathOperator{\supp}{supp}
\def\<{\langle}
\def\>{\rangle}

\title[Diagonal splittings and unimodularity]{Diagonal splittings of toric varieties and unimodularity}
\author{Jed Chou}
\email{jedchou1@illinois.edu}
\address{Department of Mathematics \\ University of Illinois \\ Urbana, IL 61801 USA}
\author{Milena Hering}
\email{m.hering@ed.ac.uk}
\address{School of Mathematics \\ University of Edinburgh \\ Edinburgh, EH9 3JZ, UK}
\author{Sam Payne}
\email{sam.payne@yale.edu}
\address{Department of Mathematics \\ Yale University \\ New Haven, CT 06511, USA}
\author{Rebecca Tramel}
\email{rtramel@illinois.edu}
\address{Department of Mathematics \\ University of Illinois \\ Urbana, IL 61801 USA}
\author{Ben Whitney}
\email{ben\_whitney@brown.edu}
\address{Division of Applied Mathematics \\ Brown University \\ Providence, RI 02912 USA}

\begin{document}
\begin{abstract}
We use a polyhedral criterion for the existence of diagonal splittings to investigate which toric varieties $X$ are diagonally split.  
Our results are stated in terms of the vector configuration given by primitive generators of the 1-dimensional cones in the fan defining $X$.  
We show, in particular, that $X$ is diagonally split at all $q$ if and only if this configuration is unimodular, and $X$ is not diagonally split at any $q$ if this configuration is not $2$-regular.
We also study implications for the possibilities for the set of $q$ at which a toric variety $X$ is diagonally split.  
\end{abstract}
\maketitle

\vspace{-20 pt} 
\section{Introduction}

Toric varieties over fields of positive characteristic are Frobenius split, and even globally $F$-regular \cite{Smith00}, and the Frobenius morphisms of toric varieties are defined over $\Z$, leading to a well-behaved notion of splittings and diagonal splittings of toric varieties at an arbitrary integer $q \geq 2$, with all of the usual formal properties \cite{BrionKumar05, diagonalsplittings}.  Here we say that a variety $X$ is diagonally split if there exists a splitting 
of $X\times X$ that is compatible with the diagonal $\Delta\subset X\times X$. 
If $X$ is diagonally split at some $q \geq 2$, then every ample line bundle on $X$ is very ample and even normally generated.   We recall that the existence of diagonal splittings on a toric variety is controlled by the vector configuration given by the primitive generators of the 1-dimensional cones in the corresponding fan, as follows.

Let $N \cong \Z^n$ be a lattice, let $\Sigma$ be a complete fan in $N_\R = N \otimes_Z \R$, and let $X = X(\Sigma)$ be the corresponding toric variety.  We write $\Sigma(1)$ for the set of primitive generators in $N$ of 1-dimensional cones in $\Sigma$, and $M = \Hom(N, \Z)$ for the dual lattice.  By \cite[Theorem~1.2]{diagonalsplittings}, $X$ is diagonally split at $q \geq 2$ if and only if the open polytope in $M_\R = M \otimes_Z R$
\[
\F_\Sigma = \{ u \in M_\R \ | \ |\< u, v \>| < 1 \mbox{ for all } v \in \Sigma(1) \}
\]
contains representatives of every equivalence class in $\frac{1}{q}M/M$, where $\Sigma(1)$ denotes the set of primitive generators of the $1$-dimensional cones in $\Sigma$.

The main purpose of this paper is to give efficient criteria for determining whether a toric variety is diagonally split in terms of basic properties of $\Sigma(1)$, and studying implications for the set of $q \geq 2$ such that $X$ is diagonally split at $q$.  Our main results are as follows.

Recall that $\Sigma(1)$ is said to be \emph{unimodular} if every maximal independent subset generates the lattice $N$.  

\begin{thm} \label{thm:unimodular}
If $\Sigma(1)$ is unimodular, then $X$ is diagonally split at $q$ for all $q \geq 2$.
\end{thm}

\begin{remark}
The vector configuration $\Sigma(1)$ is unimodular if and only if there is a choice of coordinates $N \cong \Z^n$ such that the matrix $A$ whose columns are the vectors in $\Sigma(1)$ is \emph{totally unimodular}, meaning that the determinant of any square submatrix is in $\{ -1, 0, 1 \}$.  Such matrices are well-studied from many points of view, including those of integer programming and matroid theory; see \cite[Chapters~19-20]{Schrijver86} for details.  Totally unimodular matrices have several equivalent characterizations, and there is a polynomial time algorithm for determining whether a matrix is totally unimodular.  It therefore follows from Theorem~\ref{thm:unimodular} that there is a polynomial time algorithm for determining whether $X$ is diagonally split at $q = 2$.  
\end{remark}

For $q = 2$, we have the following converse.

\begin{thm} \label{thm:q=2}
If $\Sigma(1)$ is not unimodular, then $X$ is not diagonally split at $q = 2$.
\end{thm}

Combining Theorems~\ref{thm:unimodular} and \ref{thm:q=2} gives the following equivalent characterizations of toric varieties that are diagonally split at $q = 2$.

\begin{cor}
The following are equivalent:
\begin{enumerate}
\item $\Sigma(1)$ is unimodular.
\item $X$ is diagonally split at $q$ for all $q \geq 2$.
\item $X$ is diagonally split at some even $q \geq 2$.
\item $X$ is diagonally split at $q = 2$.
\end{enumerate}
\end{cor}

We investigate cases where $X$ is not diagonally split at $q = 2$, using relaxations of the condition of unimodularity, as follows.  

Choose coordinates $N \cong \Z^n$, and let $A$ be the matrix whose columns are the vectors in $\Sigma(1)$.  Recall that $\Sigma(1)$ is unimodular if and only if every maximal nonsingular square submatrix of $A$ is invertible over $\Z$.  Following Appa and Kotnyek, we say that $\Sigma(1)$ is \emph{$k$-regular} if every maximal nonsingular square submatrix of $A$ is invertible over $\Z[\frac{1}{k}]$ \cite{AppaKotnyek04}.  Equivalently, $\Sigma(1)$ is $k$-regular if and only if, for any maximal independent subset $\{ v_1, \ldots, v_n \}$ of $\Sigma(1)$, the quotient $N/ \< v_1, \ldots, v_n\>$ is $k$-torsion. We say that $X$ is \emph{not diagonally split} if there is no $q$ such that $X$ is diagonally split at $q$.

\begin{thm} \label{thm:notsplit}
If $\Sigma(1)$ is not $2$-regular, then $X$ is not diagonally split.
\end{thm}

When $\Sigma(1)$ is $2$-regular, but not unimodular, the problem of determining the set of $q$ at which $X$ is diagonally split is more subtle.  In dimension 2, the solution is as simple and affirmative as possible.

\begin{thm} \label{thm:n=2}
If $\Sigma(1)$ is 2-regular and $\dim(X) = 2$, then $X$ is diagonally split at $q$, for all odd $q \geq 3$.
\end{thm}

Combining Theorems \ref{thm:unimodular}, \ref{thm:q=2}, and \ref{thm:n=2} gives the following classification of possibilities for the set of $q$ at which $X$ is diagonally split, when $\dim(X) = 2$.  

\begin{cor}
If $\dim(X) =2$, then exactly one of the following holds:
\begin{enumerate}
\item $\Sigma(1)$ is unimodular and $X$ is diagonally split at all $q \geq 2$.
\item $\Sigma(1)$ is 2-regular but not unimodular, and $X$ is diagonally split at $q$ if and only if $q$ is odd.
\item $\Sigma(1)$ is not 2-regular, and $X$ is not diagonally split.
\end{enumerate}
\end{cor}

\noindent Similar results hold for other special classes of 2-regular vector configurations.  For instance, a matrix is \emph{binet} if the sum of the absolute values of the entries in each column is at most 2, and binet matrices are $2$-regular  \cite[Theorem~25]{AppaKotnyek04}.  If $\Sigma(1)$ is a binet configuration, i.e., the set of column vectors in a binet matrix, then $X$ is diagonally split at all odd $q \geq 3$.  See Proposition~\ref{prop:binet}.  However, there are examples in dimensions 4 and higher where $\Sigma(1)$ is 2-regular, but $X$ is not diagonally split.  See Example~\ref{ex:notsplit}.

In Section~\ref{sec:subdiagonals}, we give a characterization of splittings of $X^n$ that are compatible with certain unions of subdiagonals, correcting an error from \cite{diagonalsplittings}. The existence of such splittings has strong consequences, for example, that every ample line bundle gives rise to an embedding that is defined by quadratic equations, or even that every section ring of an ample line bundle is Koszul, see Section~\ref{sec:subdiagonals} for further details. 

\subsection*{Acknowledgements} 
Portions of this research were carried out during an REU project supported under NSF grant DMS-1001859. We thank Arend Bayer, who co-advised this REU project and gave fundamental intellectual input.  We also thank Christian Haase.   The work of MH was partially supported by EPSRC first grant EP/K041002/1.  The work of SP was partially supported by NSF CAREER DMS-1149054.

\section{The unimodular case}

We begin with a proof of Theorem~\ref{thm:unimodular}, showing that $X$ is diagonally split at all $q \geq 2$ when $\Sigma(1)$ is unimodular.

\begin{proof}[Proof of Theorem~\ref{thm:unimodular}]
Let $a$ be a vector in $\frac{1}{q}M$.  We must show that there is a representative of the class $[a]$ in $\frac{1}{q}M/M$ in the open polytope $\F_\Sigma$, defined in the introduction.

Consider the polytope $P$ in $M_\R$, given by
\[
P =  \{ u \ | \ \lfloor \< a, v\> \rfloor \leq \<u,v\> \leq \lceil \<a,v\> \rceil, \mbox{ for all } v \in \Sigma(1). \}
\]
Note that $P$ is not empty, because it contains $a$. Since $\Sigma(1)$ is unimodular, $P$ has integer vertices \cite[Theorem~19.3]{Schrijver86}, so we can choose $x \in P \cap M$.  Then $a - x \equiv a$ in $\frac{1}{q}M/M$, and $|\<a-x, v \>| < 1$ for all $v \in \Sigma(1)$.  Therefore, $a - x$ is a representative for $[a]$ in $\F_\Sigma$, and the theorem follows. 
\end{proof}

\section{Hermite normal form and 2-regularity} \label{sec:Hermite}

To investigate cases where $\Sigma(1)$ is not unimodular, we find it helpful to consider a matrix $B$ whose columns are a maximal independent set of vectors in $\Sigma(1)$, with respect to a preferred choice of ordering of the vectors and a preferred choice of coordinates $N \cong \Z^n$.  Given such a maximal independent set $\{v_1, \ldots, v_n\}$ we first order the vectors so that $N/\<v_1, \ldots, v_r\>$ is torsion free and $N/\<v_1, \ldots, v_r, v_s\>$ is not torsion free, for all $s > r$.  

After fixing such an ordering, we can choose coordinates so that the matrix $B$ with columns $v_1, \ldots, v_n$ is in \emph{Hermite normal form}, meaning that $B$ is upper diagonal, with nonnegative integer entries, and the entries above diagonal in each column are strictly smaller than the entry on the diagonal.

Due to our choice of ordering of the vectors, the first $r$ diagonal entries of $B$ are $1$, and the rest are greater than $1$.  In other words, $B$ is a nonnegative integer matrix in the normal form
\begin{equation}  \label{eq:Hermite}
 B = \begin{pmatrix}
I_r  & C \\
0 & B'\\
\end{pmatrix},
\end{equation}
where $I_r$ is the $r\times r$ identity matrix, $B'$ is an upper triangular $(d-r)\times (d-r)$ matrix with diagonal entries at least 2,  and $B_{ij} < B_{jj}$ for $i < j$. 

Note that $\Sigma(1)$ is unimodular if and only if, for every maximal independent set of vectors, the resulting matrix $B$ is unimodular.  Furthermore, a square nonsingular matrix $B$ is unimodular if and only if its normal form is the identity matrix $I_n$.  

As mentioned in the introduction, we follow the terminology of \cite{AppaKotnyek04} and say that $\Sigma(1)$ is $k$-regular if, for any maximal independent subset $\{ v_1, \ldots, v_n \}$, the quotient $N/ \< v_1, \ldots, v_n\>$ is $k$-torsion. Note that $\Sigma(1)$ is $k$-regular if and only if, for every maximal independent subset, the resulting matrix $B$ is $k$-regular, and a square nonsingular matrix $B$ is $k$-regular if and only if it is invertible over $\Z[\frac{1}{k}]$.

We have the following characterization of $2$-regularity for $B$ in terms of its normal form.

\begin{prop}  \label{prop:2regular}
The matrix $B$ is $2$-regular if and only if $B' = 2 I_{d-r}$.
\end{prop}

\begin{proof}
Let $\{e_1, \ldots, e_n\}$ be the basis of $N$ bringing $B$ into the normal form (\ref{eq:Hermite}).  Note that the order of $e_i$ in $N/\<v_1, \ldots, v_n\>$ is divisible by the diagonal entry $B_{ii}$, with equality for all $i$ if $B'$ is diagonal.  In particular, if $B$ is $2$-regular, then every diagonal entry of $B'$ must be 2, and if $B'$ is $2I_{d-r}$, then $N/\<v_1, \ldots, v_n\>$ is generated by $2$-torsion elements, and hence $B$ is $2$-regular.

It remains to show that $B$ is not $2$-regular if every diagonal entry of $B'$ is $2$, but $B'$ is not diagonal.  To see this, choose $j>r$ such that $v_j$ is the first column vector of $B$ that contains a nonzero entry of $B'$ above the diagonal.  Then $2e_j$ is not zero in $N/\<v_1, \ldots, v_n\>$, and hence $B$ is not $2$-regular.
\end{proof}

It will also be useful to consider the intersection of $\F_\Sigma$ with coordinate subspaces compatible with the normal form of $B$, as follows.

Choose a basis $\{e_1, \ldots, e_n \}$ for $N$ with respect to which $B$ is in the normal form (\ref{eq:Hermite}), and let $\{f_1, \ldots, f_n\}$ be the dual basis for $M$.  Let $M_j$ be the sublattice of $M$ spanned by the basis vectors starting from $f_j$, so $M_j = \< f_j, \ldots, f_n\>$.

\begin{lem} \label{lem:support}
Fix $q \geq 2$, and let $[a] \in \frac{1}{q}M/M$ be the class of an element $a \in \frac{1}{q}M_j$.  If $[a]$ is represented by a point in $\F_\Sigma$, then it is represented by a point in $\F_\Sigma \cap \frac{1}{q}M_j$.
\end{lem}

\begin{proof}
Let $a + u$ be a representative for $[a]$ in $\F_\Sigma$, so $u$ is a lattice point in $M$.  Write $u = (u_1, \ldots, u_n)$, with respect to the basis $\{f_1, \ldots, f_n\}$.

If $j = 1$, then there is nothing to show.  Assume $j \geq 2$.  Then $\<a + u, v_1\> = u_1$.  Since $u_1$ is an integer and $a + u$ is in $\F_\Sigma$, it follows that $u_1$ is zero.  Similarly, by an induction on $i$, we conclude that $\<a + u, v_i \> = B_{ii} u_i$, and hence $u_i = 0$, for $i < j$.  This shows that $u \in M_j$, as required.
\end{proof}

\section{A converse theorem for $q = 2$.}

We have already shown that if $\Sigma(1)$ is unimodular then $X$ is diagonally split for all $q \geq 2$.  We now prove the converse for $q = 2$.

\begin{proof}[Proof of Theorem~\ref{thm:q=2}]
Suppose $\Sigma(1)$ is not unimodular.  We must show that $X$ is not diagonally split at $q = 2$.

Choose a maximal independent subset $\{v_1, \ldots, v_n\}$ that does not generate $N$.  Reorder the vectors and choose a basis $\{e_1, \ldots, e_n\}$ on $N$ so that the matrix $B$ whose columns are $v_1, \ldots, v_n$ is in the normal form (\ref{eq:Hermite}).  Let $\{f_1, \ldots, f_n\}$ be the dual basis for $M$.  By Lemma~\ref{lem:support}, if the class $[f_n/2]$ in $\frac{1}{2}M/M$ is represented in $\F_\Sigma$, then it is represented by an odd multiple of $f_n/2$.  But this is impossible, since $\<f_n , v_n \> = B_{nn}$, which is at least $2$.
\end{proof}

\section{Configurations that are not 2-regular}

We now consider the case where $\Sigma(1)$ is not $2$-regular, and prove Theorem~\ref{thm:notsplit}, showing that $X$ is not diagonally split in this case.

\begin{proof}[Proof of Theorem~\ref{thm:notsplit}]
Choose a maximal independent set $\{v_1, \ldots, v_n \}$ in $\Sigma(1)$, such that $N/ \<v_1, \ldots, v_n\>$ is not $2$-torsion.  After reordering the vectors and choosing coordinates on $N$, we may assume that the matrix $B$ with columns $v_1, \ldots, v_n$ is in the normal form (\ref{eq:Hermite}).  By Proposition~\ref{prop:2regular}, either $B$ has a diagonal entry that is greater than 2, or the lower right square matrix $B'$ is not diagonal.

Suppose the diagonal entry $B_{jj}$ is greater than 2.  Let $\{e_1, \ldots, e_n\}$ be the chosen basis for $N$, with $\{f_1, \ldots, f_n\}$ the dual basis for $M$, and fix $a = \frac{1}{q}\lfloor \frac{q}{2}\rfloor f_j$.  By Lemma~\ref{lem:support}, if the class $[a] \in \frac{1}{q}M/M$ is represented by a point in $\F_\Sigma$, then it is represented by a point in $F_\Sigma \cap \frac{1}{q}M_j$.  Any such point $u$ is of the form $a + a_j f_j + \cdots + a_n f_n$, for some integers $a_j, \ldots, a_n$, and hence $\<u, v_j\>$ is in the set $ B_{jj} (\frac{1}{q}\lfloor \frac{q}{2}\rfloor + \Z)$.  Since $\frac{1}{q}\lfloor \frac{q}{2} \rfloor$ is in the interval $[\frac{1}{3}, \frac{1}{2}]$ and $B_{jj} \geq 3$, it follows that $|\<u,v_j\> |\geq 1$, and hence $u$ is not in $\F_\Sigma$.

It remains to consider the case where all diagonal entries of $B'$ are 2, and $B'$ is not diagonal.  Choose $j$ as small as possible so that $B'$ contains a nonzero off-diagonal entry in the $j$th column of $B$, and choose $i$ as small as possible, with $j$ fixed, so that $B_{ij}$ is such an entry.

Let $a = \frac{1}{q}(f_{i}+\lfloor \frac{q}{2} \rfloor f_{j})$, and suppose the class $[a] \in \frac{1}{q}M/M$ is represented by a point $u$ in $\F_\Sigma$.  By Lemma~\ref{lem:support}, any such point $u$ is of the form $a + a_i f_i + \cdots + a_n f_n$, for some integers $a_i, \ldots, a_n$.  Note that $\<u, v_i\> = 2(\frac{1}{q} + a_i)$, so $|\<u,v_i\>| < 1$ implies that $a_i = 0$.  Similarly, $\<u,v_\ell\> = 2 a_\ell$, and hence $a_\ell = 0$, for $i < \ell < j$.  We then compute that $\<u, v_j\> = \frac{1}{q} + \frac{2}{q}\lfloor \frac{q}{2} \rfloor + 2 a_j$.  By Theorem~\ref{thm:q=2}, we may assume that $q$ is odd, and conclude that $\<u,v_j\> = 1 + 2a_j$.  In particular, $|\<u,v_j\>|$ cannot be less than 1, so $u$ is not in $\F_\Sigma$, and hence $X$ is not diagonally split at $q$.
\end{proof}

\section{Partial results in the 2-regular case}

In this section, we investigate possibilities for the set of $q$ at which $X$ is diagonally split when $\Sigma(1)$ is 2-regular but not unimodular.  We show that if $\Sigma(1)$ is as small as possible, in the sense that it is contained in a basis for $N_\R$ and its negatives, or the dimension $n$ is 2, or when $\Sigma(1)$ forms the columns of a binet matrix, then $X$ is diagonally split for all odd $q$.  Example~\ref{ex:notsplit} shows that, in dimensions 4 and higher, $X$ may not be diagonally split, even though $\Sigma(1)$ is $2$-regular.

We first consider the case where $\Sigma(1)$ is as small as possible.

\begin{prop}
\label{prop:2n}
Suppose $\Sigma(1)$ is $2$-regular and contained in $\{\pm v_1, \ldots, \pm v_n\}$, where $\{v_1, \ldots, v_n\}$ is a basis for $N_\R$.  Then $X$ is diagonally split at $q$, for all odd $q \geq 3$.
\end{prop}

\begin{proof}
After permuting the vectors, we can choose coordinates $N \cong \Z^n$ such that the matrix $B$ whose column vectors are $v_1, \ldots, v_n$ is in the normal form (\ref{eq:Hermite}).  Since $\Sigma(1)$ is $2$-regular, the matrix $B'$ is $2I_{n-r}$, by Proposition~\ref{prop:2regular}.

Assume $q$ is odd.  Then we can represent any class in $\frac{1}{q}\Z^n/\Z^n$ uniquely by a vector $a = (a_1, \ldots, a_n)$ where $a_i \in (-1,1)$ and $q a_i$ is even.  Then $|\<a, v_j\>| < 1$ for $1 \leq j \leq r$.  For $j > r$, we have $\<a, v_j \> \in \frac{2}{q} \Z$.  In particular, $\<a, v_j\>$ is not an odd integer, so there is a unique integer $u_j$ such that $|\<a,v_j\> -2u_j| < 1$.  Then
\[
a + (0, \ldots, 0, u_{r+1}, \ldots, u_n)
\]
is in $\F_\Sigma$ and represents the class $[a] \in \frac{1}{q}M/M$, so $X$ is diagonally split at $q$.
\end{proof}

A matrix is called \emph{binet} if the sum of the absolute values of the entries in each column is at most 2, and binet matrices are $2$-regular \cite[Theorem~25]{AppaKotnyek04}.  We say that $\Sigma(1)$ is binet if there is a choice of coordinates $N \cong \Z^n$ such that the matrix whose columns are the vectors in $\Sigma(1)$ is binet.

\begin{prop} \label{prop:binet}
If $\Sigma(1)$ is binet, then $X$ is diagonally split at $q$, for all odd $q \geq 3$.
\end{prop}

\begin{proof}
Choose coordinates $N \cong \Z^n$ so that the sum of the absolute values of the coordinates of each vector in $\Sigma(1)$ is at most 2.  Then $\F_\Sigma$ contains the open cube with vertices $(\pm\frac{1}{2}, \ldots, \pm \frac{1}{2})$, and hence contains all points of the form $(\frac{a_1}{q}, \ldots, \frac{a_n}{q})$ where $|a_i| < \frac{q}{2}$.  These represent all equivalence classes in $\frac{1}{q}M/M$ when $q$ is odd, and the proposition follows. 
\end{proof}

Next, we consider the case where the dimension is as small as possible.

\begin{proof}[Proof of Theorem~\ref{thm:n=2}]
Suppose the dimension $n$ is $2$, and $\Sigma(1)$ is $2$-regular.  We will show that $X$ is diagonally split at all odd $q$, by classifying the possibilities for $\Sigma(1)$ and showing that they are all binet.

We first consider the cases where $\Sigma(1)$ does not contain a basis for $N$.  Fix two independent vectors in $\Sigma(1)$, and choose coordinates so that the corresponding matrix is in Hermite normal form.  In these coordinates, the two vectors must be $(1,0)$ and $(1,2)$.  Any other vector $v$ that is not equal to these two or their negations must be of the form $(a, \pm 2)$.  The condition that $v$ and $(1,2)$ generate a sublattice of index 2 guarantees that $a$ is even, which is impossible, since $v$ must be primitive.  We conclude that $\Sigma(1) \subset \{ \pm (1,0), \pm (1,2) \}$, which is binet with respect to the basis $\{(1,1), (0,1)\}$.

Therefore, we may assume $\Sigma(1)$ contains a basis for $N$, and take these as coordinates.  Since $\F_\Sigma$ is cut out by absolute values of pairings with vectors in $\Sigma(1)$, we may restrict attention to configurations of vectors whose first nonzero coordinate is positive.  Then $2$-regularity implies that the remaining vectors are a subset of $$\{(1,1), (1,-1), (1,2), (1,-2), (2,1), (2,-1)\}.$$  Furthermore, $\Sigma(1)$ can contain at most one vector from $S = \{(1,2), (1,-2), (2,1), (2,-1)\}$, since the quotient of $\Z^2$ by any two of these contains $3$-torsion or $4$-torsion.  Similarly, if $\Sigma(1)$ contains $(1,1)$ and $(1,-1)$, then it cannot contain any vector from $S$, since each vector with $S$, together with either $(1,1)$ or $(1,-1)$ generates an index 3 sublattice.  If $\Sigma(1)$ does not contain any vector from $S$, then it is binet in the given coordinates.  

We therefore assume $\Sigma(1)$ contains exactly one element of $S$.  After permuting the coordinates, we may suppose $S$ contains $(1,2)$ or $(1,-2)$.  Since adding vectors can only diminish the set of $q$ at which $X$ is diagonally split, we may assume $\Sigma(1)$ contains either $(1,1)$ or $(1,-1)$, as well.  This leaves two cases, namely $\Sigma(1) = \{ (1,0), (0,1), (1,1), (1,2) \}$ and $\Sigma(1) = \{ (1,0), (0,1), (1,-1), (1,-2) \}$, and both differ only by a change of coordinates from $\{ (1,0), (0,1), (1,1), (1,-1) \}$.  We conclude that $\Sigma(1)$ is binet and hence $X$ is diagonally split at all odd $q \geq 3$.
\end{proof}

The following is a 4-dimensional example where $\Sigma(1)$ is $2$-regular, but $X$ is not diagonally split.  We do not know whether such examples exist in dimension $3$.

\begin{ex}  \label{ex:notsplit}
Consider a fan $\Sigma$ in $\R^4$ whose $1$-dimensional cones are generated by
$$\Sigma(1)=\{e_1, e_2, e_3 + e_4, e_3- e_4, e_1 + e_2 - e_3, e_1-e_2 + e_3, -e_1 + e_2 + e_3 \}.$$  We claim that $X$ is not diagonally split. Note that $\Sigma(1)$ is not unimodular so, by Theorem~\ref{thm:q=2}, $X$ is not diagonally split at $q$ when $q$ is even.  It remains to check that $X$ is not diagonally split at any odd $q \geq 3$.  Note that every maximal independent subset generates either $\Z^4$ or a sublattice of index $2$, so $\Sigma(1)$ is $2$-regular, and Theorem~\ref{thm:notsplit} does not apply.  Nevertheless, we verify that $X$ is not diagonally split at any odd $q$ as follows.

Let $q \geq 3$ be odd.  We claim that the class of $a = \frac{1}{q}(\lfloor \frac{q}{2} \rfloor (e_1 + e_2) - e_3 + e_4)$ in $\frac{1}{q}\Z^4/ \Z^4$ is not represented by any point in $\F_\Sigma$.   To see this, suppose that $a + (a_1, a_2, a_3, a_4)$ is in $\F_\Sigma$, for some integers $a_1, \ldots, a_4$.  Pairing with $e_3 + e_4$ and $e_3 - e_4$ shows that $a_3$ and $a_4$ must vanish.  Pairing with $e_1$ and $e_2$ then shows that $a_1$ and $a_2$ are each either $0$ or $-1$, and pairing with the remaining vectors eliminates these four possibilities.
\end{ex}

\section{Compatibly split subdiagonals} \label{sec:subdiagonals}
 
As mentioned in the introduction, there is significant geometric interest in knowing whether a variety $X$ is diagonally split because every ample line bundle on a diagonally split variety is very ample and gives rise to a projectively normal embedding.  There are similar reasons for interest in compatible splittings of unions of subdiagonals in higher products of $X$.  Indeed, if $X \times X \times X$ is split compatibly with the union of $\Delta \times X$ and $X \times \Delta$ then every ample line bundle on $X$ gives rise to an embedding that is normally presented, i.e., the homogeneous ideal of the image is generated by quadrics.  Linearity of subsequent steps in the minimal free resolution of the ground field 
over the homogeneous coordinate ring of the image are guaranteed by splittings of $X^n$ compatible with the union of the higher subdiagonals
\[
\Delta_i = X^{i-1} \times \Delta \times X^{n-i-1},
\]
for $1 \leq i < n$.  In particular, the homogeneous coordinate ring is Koszul if $X^n$ is split compatibly with $\Delta_1 \cup \cdots \cup \Delta_{n-1}$, for all $n$.

Payne mistakenly stated that if a toric variety $X$ is diagonally split, then $X^n$ is split compatibly with $\Delta_1 \cup \cdots \cup \Delta_{n-1}$ for all $n$ \cite[Theorem~1.3]{diagonalsplittings}.  The error in the proof occurs in the middle of the second paragraph, with the false assertion that a certain explicit splitting $\pi$ is compatible with $\Delta_1 \cup \cdots \cup \Delta_{n-1}$.  Indeed, whenever $n \geq 3$, the splitting described there fails to satisfy the necessary and sufficient condition for compatibility with $\Delta_1 \cup \cdots \cup \Delta_{n-1}$ given in Theorem~\ref{thm:compatiblewithDelta}.

\begin{remark}
Note that the error in \cite[Theorem~1.3]{diagonalsplittings} also affects \cite[Theorems~1.1 and 1.3]{rootpolytopes}.  The main arguments in the latter paper show correctly that if $\Sigma(1)$ is contained in a root system of type $A$, $B$, $C$, or $D$ then the toric variety $X$ is diagonally split at $q$ for all odd $q \geq 3$.  It follows that any lattice polytope whose facet normals is in one of these root systems is normal, as are Cayley sums of polytopes whose Minkowski sum is such a polytope.  However, it does not follow that these polytopes are Koszul.  The Koszulness of lattice polytopes whose facet normals are contained in a root system of type $A$ is known, by \cite{BrunsGubeladzeTrung97}.  The corresponding statement for root systems of type $B$, $C$, or $D$ is an open problem.
\end{remark}

The following is an example of a diagonally split toric variety $X$ such that $X \times X \times X$ is not split compatibly with the union of $\Delta \times X$ and $X \times \Delta$.

\begin{ex}\label{ex:Birkhoff}
The Birkhoff polytope $B_n$ is the convex hull of the $n \times n$ permutation matrices.  It is a lattice polytope of dimension $n^2 - 2n + 1$ cut out by inequalities coming from a totally unimodular matrix \cite[\S19]{Schrijver86}.  In particular, if $X(\Sigma_n)$ is the toric variety corresponding to $B_n$ then $\Sigma_n(1)$ is unimodular, and hence $X(\Sigma_n)$ is diagonally split at $q$, for all $q \geq 2$, by Theorem~\ref{thm:unimodular}.

However, as noted by Haase and Paffenholz \cite{HaasePaffenholz09}, for $n =3$, the polytope $B_3$ corresponds to an embedding of $X(\Sigma_3)$ as the cubic hypersurface $x_0 x_1 x_2 = x_3 x_4 x_5$ in the projective space $\mathbf{P}^5$.  Since the homogeneous ideal of this embedding is not generated by quadrics, it follows that $X \times X \times X$ has no splitting compatible with $(\Delta \times X) \cup (X \times \Delta)$.
\end{ex}

In the remainder of this section, we characterize the splittings of $X^n$ that are compatible with $\Delta_1 \cup \cdots \cup \Delta_{n-1}$.  Recall that multiplication by an integer $q \geq 2$ on $N_{\bR}$ preserves the fan $\Sigma$ and hence induces a morphism $F  \colon X\to X$, which agrees with the absolute Frobenius morphism when $q$ is prime and $k$ is the field with $q$ elements.  

A \emph{splitting} of $X$ at $q$ is an $\mathcal{O}_X$-module map $\pi: F_*\mathcal{O}_X \rightarrow \mathcal{O}_X$ such that the composition $\pi \circ F_*$ is the identity on $\mathcal{O}_X$.  Such a splitting is compatible with a subvariety $Z \subset X$ if $\pi$ maps $F_*(I_Z)$ into $I_Z$, and $X$ is \emph{diagonally split} if there is a splitting of $X \times X$ compatible with the diagonal $\Delta$.

Following \cite[Section~2]{diagonalsplittings}, we recall that the global sections of $F_*\mathcal{O}_T$ are generated by monomials $x^b$ for $b \in \frac{1}{q}M$, and a basis for $\Hom(F_* \mathcal{O}_T, \mathcal{O}_T)$ is given by maps $\pi_a$ for $a \in \frac{1}{q}M$, where
\[
\pi_a (x^b) = \left\{
\begin{array}{ll} x^{a+b} & \mbox{ if } a+b \in M, \\
			 0 & \mbox{ otherwise,}
			 \end{array}
			 \right.
\]
and $\sum c_a \pi_a$ extends to a map from $F_*\mathcal{O}_X$ to $\mathcal{O}_X$ if and only if $\supp(\pi) = \{ a \ | \ c_a \neq 0\}$ is contained in the open polytope
\[
P_{-K}^\circ = \{ u \in M_\R \ | \ \<u, v\> < 1 \mbox{ \ for all \ } v \in \Sigma(1)\}.
\]
Each map from $F_*\mathcal{O}_X$ to $\mathcal{O}_X$ is determined by its restriction to $F_* \mathcal{O}_T$, so $\Hom(F_*\mathcal{O}_X, \mathcal{O}_X)$ has a basis given by $P_{-K}^\circ \cap \frac{1}{q}M$.  Furthermore, a map $\sum c_a \pi_a$ is a splitting if and only if $c_0 = 1$.  It follows that the space of splittings compatible with a given subscheme is an affine hyperplane in $\Hom(F_*\mathcal{O}_X, \mathcal{O}_X)$.

The main result of \cite{diagonalsplittings} says that $X \times X$ is diagonally split if and only if $\F_\Sigma$, which is equal to $P_{-K}^\circ \cap -P_{-K}^\circ$, contains representatives of every equivalence class in $\frac{1}{q}M/M$.  It generalises as follows. 

Let $L \subset M^n$ be the sublattice of tuples $(u_1, \ldots, u_n)$ such that $u_1 + \cdots + u_n = 0$.  Note that $L = L_1 \oplus \cdots \oplus L_{n-1}$, where $L_i \subset L$ is the sublattice where $u_j = 0$ for $j \not \in \{ i, i+1 \}$.

\begin{thm} \label{thm:compatiblewithDelta}
A splitting of $X^n$ given by $$ \pi = \sum_{a \in \frac{1}{q} M^n} c_a \pi_a$$ is compatible with $\Delta_1\cup \cdots \cup \Delta_{n-1}$ if and only if, for any $b$ and $b'$ in $\frac{1}{q}M^n$ such that $b \equiv b'$ mod $\frac{1}{q}L_i$, and any class $[u] \in M^n/L_i$, we have
\[
\sum_{\substack {a + b \, \in \, M^n \\ (a + b) \, \mathrm{mod} \, L_i \, = \, [u]}} c_a = \sum_{\substack {a' + b' \, \in \, M^n \\ (a' + b') \, \mathrm{mod} \, L_i \, = \, [u]}} c_{a'},
\]
for $1 \leq i \leq n-1$.
\end{thm}

\begin{remark}
If $X$ is diagonally split, then for all $n$ and all $i$, $\Delta_i$ is compatibly split in $X^n$. However, the Birkhoff polytope 
Example~\ref{ex:Birkhoff} shows that 
there need not be a splitting that is simultaneously compatible for all $i$. 
\end{remark}

\begin{remark}
Note that $0$ is the only lattice point in 
$(P_{-K}^\circ)^n$. Thus, setting $[u]=0$ and $b=0$ in Theorem~\ref{thm:compatiblewithDelta} shows that a necessary condition for the existence of such a splitting is that  
$(P_{-K}^\circ)^n$ contains a representative of every equivalence 
class of $\frac{1}{q}L_i/L_i$. In particular, when $n=2$, $P_{-K}^\circ
\times P_{-K}^\circ$, must contain a representative of every equivalence class of $\frac{1}{q}L/L$, which is equivalent to the condition in
\cite[Theorem 1.2]{diagonalsplittings}.
\end{remark}

\begin{proof}[Proof of Theorem~\ref{thm:compatiblewithDelta}]
A splitting is compatible with $\Delta_1 \cup \cdots \cup \Delta_{n-1}$ if and only if it is compatible with $\Delta_i$, for each $1 \leq i \leq n-1$ \cite[Proposition~1.2.1]{BrionKumar05}.  Therefore, it will suffice to show that $\pi$ is compatible with $\Delta_i$ if and only if, for any $b$ and $b'$ in $\frac{1}{q}M^n$ such that $b \equiv b'$ mod $\frac{1}{q}L_i$, and any $[u] \in M^n/L_i$, we have
\[
\sum_{\substack {a + b \, \in \, M^n \\ (a + b) \, \mathrm{mod} \, L_i \, = \, [u]}} c_a = \sum_{\substack {a' + b' \, \in \, M^n \\ (a' + b') \, \mathrm{mod} \, L_i \, = \, [u]}} c_{a'}.
\]

The splitting $\pi$ is compatible with $\Delta_i$ if and only if the restriction of $\pi$ to $T^n$ is compatible with $\Delta_i \cap T^n$ \cite[Lemma~1.1.7]{BrionKumar05}.  The coordinate ring of $T^n$ is generated by monomials $x^u$ for $u \in M^n$, and the ideal $I$ defining $\Delta_i \cap T^n$ is generated by differences $x^{u} - x^{u'}$, for $u, u' \in M^n$  such that $u \equiv u' \ \mathrm{mod} \ L_i$.

It follows that $F_*I$ is generated by differences $x^{b} - x^{b'}$, for $b, b' \in \frac{1}{q}M^n$  such that $b \equiv b' \ \mathrm{mod} \ \frac{1}{q}L_i$, and $\pi$ is compatible with $\Delta_i$ if and only if it maps each of these generators to a function that vanishes on $\Delta_i \cap T^n$.  The coordinate ring of $\Delta_i \cap T^n$ is the group ring of $M^n / L_i$, so the restriction of $\pi(x^b - x^{b'})$ to $\Delta_i \cap T^n$ can be expressed uniquely as a linear combination of monomials $x^{[u]}$ for $[u] \in M^n/L_i$.  Since restriction from the coordinate ring of $T^n$ is induced by the projection $M^n \rightarrow M^n / L_i$, the coefficient of $x^{[u]}$ in $\pi(x^b)$ is
\[
\sum_{\substack {a + b \, \in \, M^n \\ (a + b) \, \mathrm{mod} \, L_i \, = \, [u]}} c_a,
\]
and the theorem follows.
\end{proof}

\bibliographystyle{alpha}
\bibliography{math}

\end{document}